\documentclass[12pt,leqno]{amsart}
\newcommand{\version}{June 14, 2026}
\usepackage{amssymb}
\usepackage{esint}
\usepackage{mathtools}
\usepackage{hyperref}
\usepackage{cleveref}
\makeatletter
\newcommand*{\rom}[1]{\expandafter\@slowromancap\romannumeral #1@}
\makeatother
\title[Sobolev extensions, interpolation inequalities and consequences]
{Sobolev extensions, interpolation inequalities\\ and consequences}
\author{Rupert L. Frank}
\address[R.~L.~Frank]{Mathe\-matisches Institut, Ludwig-Maximilians
Universit\"at M\"unchen, The\-resienstr.~39, 80333 M\"unchen, Germany, and
Munich Center for Quantum Science and Technology, Schel\-ling\-str.~4,
80799 M\"unchen, Germany, and Mathematics 253-37, Caltech, Pasa\-de\-na,
CA 91125, USA} \email{r.frank@lmu.de}
\author{Pekka Koskela}
\address[P.~Koskela]{Department of Mathematics and Statistics, University
of Jyv\"askyl\"a, P.O. Box 35 (MaD), FI-40014, University of Jyv\"askyl
\"a, Finland} \email{pekka.j.koskela@jyu.fi}
 \author{Riddhi Mishra}
\address[R.~Mishra]{Department of Mathematics and Statistics, University
of Jyv\"askyl\"a, P.O. Box 35 (MaD), FI-40014, University of Jyv\"askyl
\"a, Finland} \email{riddhi.r.mishra@jyu.fi}

plus 6pt minus 12pt
\newtheorem{theorem}{Theorem}
\newtheorem{lemma}[theorem]{Lemma}
\newtheorem{corollary}[theorem]{Corollary} \newtheorem{proposition}
[theorem]{Proposition}

\theoremstyle{definition}
\newtheorem{remark}[theorem]{Remark}

\newcommand{\barint}{
\rule[.036in]{.12in}{.009in}\kern-.16in \displaystyle\int }
\newcommand{\barcal}{\mbox{$ \rule[.036in]{.11in}{.007in}\kern-.128in\int
$}}
\newcommand{\diam}{\mathop{\mathrm {diam}}}

\def\diam{\operatorname{diam}}

\def\supp{{\rm supp\,}} \def\dist{{\rm dist\,}}
\def\loc{{\rm loc\,}}

 \renewcommand{\epsilon}{\varepsilon}

 \renewcommand{\phi}{\varphi} \newcommand
{\R}{\mathbb{R}}

 \DeclareMathOperator{\tr}

\date{\version}
\thanks{The first author acknowledges partial support through US National
Science Foundation grant DMS-1954995, as well as through the German
Research Foundation through EXC-2111-390814868 and TRR 352-Project-ID
470903074. The second and third authors have been supported by the Academy
of Finland via Centre of Excellence in Analysis and Dynamics Research
(Project number 323960).}
\begin{document}
\begin{abstract}
 We prove Sobolev interpolation inequalities on extension domains that
have a form reminiscent of the corresponding whole-space inequalities.
This form is crucial in certain applications, which we discuss as well.
The technical key ingredient is the notion of a Lebesgue $W^{1,p}$-
extension domain, which we introduce here, and our proof that, for $1<p<
\infty$, any $W^{1,p}$-extension domain is a Lebesgue $W^{1,p}$-extension
domain. \end{abstract}

\maketitle
\section{Introduction and main results}
\subsection{A Sobolev interpolation inequality}
Let $\Omega\subset\R^n$ be a domain (a connected open set) and $1\leq p
\leq \infty$. Recall that the Sobolev space $W^{1,p}(\Omega)$ consists of
all $L^p$-integrable functions on $\Omega$ whose first order
distributional derivatives belong to $L^p(\Omega)$. 
the norm 
We say that $E$ is a $W^{1,p}(\Omega)$-extension operator if
\begin{equation}
 E:W^{1,p}(\Omega)\to W^{1,p}(\mathbb{R}^n),
\end{equation}
i.e., for every $u\in W^{1,p}(\Omega)$ there exists an $E(u)\in W^{1,p}
(\mathbb{R}^n)$ with $E(u)\big|_\Omega\equiv u$. We say that the extension
operator is bounded if there is a constant $C\geq 1$ such that \[\|E(u)\|_
{W^{1,p}(\mathbb{R}^n)}\leq C\|u\|_{W^{1,p}(\Omega)}
\qquad\text{for all}\ u\in W^{1,p}(\Omega) \,.
\] A domain $\Omega$ is said to be a $W^{1,p}$-extension domain if there
is a bounded $W^{1,p}(\Omega)$-extension operator. We emphasize that we do
not assume that $E$ is linear.

We are interested in Sobolev interpolation inequalities for extension
domains. The model inequality on $\R^n$ is of the form \begin
{equation}\label{eq1.2}
 \| u \|_{L^r(\R^n)} \leq C \| \nabla u \|_{L^{p}(\R^n)}^
 \theta \|u \|_{L^s(\R^n)}^{1-\theta}
\end{equation}
for $u\in L^s(\R^n)\cap L^1_\loc(\R^n)$ with $\nabla u \in L^p(\R^n)$.
Here the exponents satisfy $1<p<\infty$, $0<s<\infty$ and \begin{equation}
 \label{eq:assr}
 \begin{cases}
 s\leq r \leq \frac{np}{n-p} & \text{if}\ p< n \,,\\
 s\leq r<\infty & \text{if}\ p=n>1 \,,\\
 s\leq r \leq \infty & \text{if}\ p>n \ \text{or}\ p=n=1 \,,
 \end{cases}
\end{equation}
and $\theta\in[0,1]$ satisfies
\begin{equation}
 \label{eq:asstheta}
 \theta \left( \frac{1}{p} - \frac{1}{n} \right) + \left(1-\theta\right)\,
\frac {1}{s}
 = \frac{1}{r} \,.
\end{equation} Moreover, the constant $C$ in \eqref{eq1.2} depends only on
$n$, $p$, $s$ and $r$.

A feature of \eqref{eq1.2} that is crucial in several applications is that
it involves a $\theta$-geometric mean of $ \| \nabla u \|_{L^{p}(\R^n)}$
and $\|u \|_{L^s(\R^n)}$. Intuitively, this means that only a fraction $
\theta$ of the gradient, and not the whole gradient, is necessary to
control an $L^r$-norm.

In view of the elementary inequality $a^\theta b^{1-\theta} \leq \theta a
+ (1-\theta) b$, inequality \eqref{eq1.2} with the $\theta$-geometric mean
implies a corresponding inequality with the arithmetic mean. Moreover, in
$\R^n$, a simple scaling argument shows that the inequality with the $
\theta$-geometric mean and with the arithmetic mean are equivalent. This
scaling argument, however, is no longer available when the whole space is
replaced by an extension domain and a straightforward application of the
extension property only yields the interpolation inequality with the
arithmetic mean, viz. \begin{equation} \label{eq:sobintadditive}
 \| u \|_{L^r(\Omega)} \leq C \left( \| u \|_{W^
 {1,p}(\Omega)} + \|u \|_{L^s(\Omega)} \right).
\end{equation}
This, however, is not enough in certain applications, some of which we
discuss below.

The following is our main result. It provides a Sobolev interpolation
inequality with the $\theta$-geometric mean on extension domains.

\begin{theorem}\label{sobint}
Let $1<p<\infty$, let $\Omega\subset\R^n$ be a $W^{1,p}$-extension domain
and let $0<s<\infty$. Then for all $r$ as in \eqref{eq:assr} one has
\begin{equation}
 \label{eq:embedding}
 W^{1,p}(\Omega)\cap L^s(\Omega) \subset L^r(\Omega)
\end{equation}
and for all $u\in W^{1,p}(\Omega)\cap L^s(\Omega)$ we have
\begin{equation}
 \label{eq:sobintext}
 \| u \|_{L^r(\Omega)} \leq C \| u \|_{W^{1,p}(\Omega)}^\theta \|u \|_
{L^s
 (\Omega)}^{1-\theta} \,,
\end{equation}
where $\theta\in[0,1]$ is given by \eqref{eq:asstheta} and the constant $C
$ only depends on $\Omega$, $p$, $s$ and $r$.
\end{theorem}
\begin{remark} (a) The embedding \eqref{eq:embedding} is well-known,
but a direct proof based on the extension property only gives the
`arithmetic mean' bound \eqref{eq:sobintadditive}, see e.g. \cite{WYZ24}. The thrust of our
result is the `$\theta$-geometric mean' bound \eqref{eq:sobintext}. This $
\theta$-geometric mean bound is well-known when $\Omega$ has uniformly
Lipschitz continuous boundary, based on the explicit construction of an
extension operator in this case. In contrast, our theorem is valid for
\emph{any} extension domain, irrespectively of how the extension operator
is constructed.
 
 (b) The specific value of $\theta$ in Theorem \ref{sobint} is the
same as in the whole-space inequality \eqref{eq1.2} and the latter is
uniquely determined by a scaling argument. Note that if the equality $$
\| u \|_{L^r(\Omega)} \leq C \| u \|_{W^{1,p}
 (\Omega)}^\delta \|u \|_{L^s(\Omega)}^{1-\delta} $$ holds with
some exponent $\delta\in[0,1]$ and $r>s$, then necessarily $$ \delta
\geq \theta \,. $$ Indeed, take $u\in C^1_c(\R^n)$ and apply the assumed 
inequality to $u((\cdot -x_0)/\epsilon)$ for some point $x_0\in\Omega$.
 Comparing the powers of both sides as $\epsilon\to 0$, we find $$
\delta
 \left( \frac1p - \frac1n \right) + \left(1-\delta\right)\, \frac 1s
\leq
 \frac1r \,, $$ from which we easily conclude the claimed inequality $
 \delta\geq\theta$.
 
 (c) The proof of Theorem \ref{sobint} is easy when $p< n$, but when
$p\geq n$
 it relies on our new Theorem \ref{thm1} below about Lebesgue $W^
{1,p}$-extension operators. While the proof that we give for $p\geq n$
also works for $p< n$, we now present a simple direct argument in the
latter case and
 explain the difficulty of extending it to $p\geq n$.
 For $p<n$ we set $p^*:=np/(n-p)$. Then the extension property implies
that
 $$ \| u \|_{L^{p^*}(\Omega)} \leq C \| u \|_{W^{1,p}(\Omega)} \,. $$
(This is the limiting case of Theorem \ref{sobint}, corresponding to $
\theta
 =1$.) Now given $r,s,\theta$ as in the theorem we have, by H\"older's inequality, $$ \| u \|_{L^r(\Omega)} \leq \| u \|_{L^{p^*}(\Omega)}^
 {\theta} \| u \|_{L^s(\Omega)}^{1-\theta} \,. $$ Combining this with the
 bound on $\| u \|_{L^{p^*}(\Omega)}$ implies the inequality in the
 theorem.
 For comparison, let us see how a corresponding argument fails to
prove the
 theorem in case $p\geq n$. For simplicity we give restrict ourselves
to
 $p>n$. By the extension property, we have $$ \| u \|_{L^\infty
(\Omega)}
 \leq C \| u \|_{W^{1,p}(\Omega)} \,. $$ Combining this with the bound
$$
 \| u \|_{L^r(\Omega)} \leq \| u \|_{L^\infty(\Omega)}^{1-s/r} \| u \|
_{L^s
 (\Omega)}^{s/r} $$ gives $$ \| u \|_{L^r(\Omega)} \leq \| \nabla
u \|_{W^
 {1,p}(\Omega)}^{1-s/r} \| u \|_{L^s(\Omega)}^{s/r} \,. $$ This
is a bound
 as in the theorem, but with the (optimal) power $\theta$ replaced by
the larger power $1-s/r$. The argument for $p=n$ is similar, but
only leads to
 powers even greater than $1-s/r$.
 
 (d) In the case $p=n>1$ our proof (based on \cite[Theorem 3.4]{BCLS})
also gives exponential integrability of $u$ and it yields a bound on the
behavior of the constant $C$ as $r\to\infty$. For the sake of conciseness
we refrain from stating this explicitly. \end{remark}
Applying the inequality in Theorem \ref{sobint} to the function $u-u_
\Omega$ and using the Poincar\'e inequality, we arrive at the following
consequence.
\begin{corollary}
Let $1<p<\infty$, let $\Omega\subset\R^n$ be a $W^{1,p}$-extension domain
of finite measure and let $0<s<\infty$. Then, for $r$ and $\theta$ as in
Theorem \ref{sobint} and any $u\in W^{1,p}(\Omega)\cap L^s(\Omega)$,
$$
\| u - u_\Omega \|_{L^r(\Omega)} \leq C \| \nabla u \|_{L^{p}(\Omega)}^
\theta \|u - u_\Omega \|_{L^s(\Omega)}^{1-\theta} \,.
$$
\end{corollary}
\subsection{Lebesgue and Lebesgue--Dirichlet extension operators}
 We now present the technical core of the proof of Theorem \ref{sobint}.
 We say that a bounded $W^{1,p}(\Omega)$-extension operator $E$ is a \emph
{Lebesgue $W^{1,p}(\Omega)$-extension operator} if for every $1<q\leq
\infty$ there is a constant $C_q$ such that  \begin{equation}\label{lqeq}
 \|E(u)\|_{L^q(\mathbb{R}^n)}\leq C_q\|u\|_{L^q(\Omega)}
 \qquad\text{for all}\ u\in L^{q}(\Omega) \,.
 \end{equation}
We say that a bounded $W^{1,p}(\Omega)$-extension operator $E$ is a \emph
{Lebesgue--Dirichlet $W^{1,p}(\Omega)$-extension operator} if it is a
Lebesgue $W^{1,p}(\Omega)$-extension operator and if there is a constant $C$ such that \begin{equation}\label{greq}
 \|\nabla E(u)\|_{L^p(\mathbb{R}^n)}\leq C\|\nabla u\|_{L^p(\Omega)}
 \qquad\text{for all}\ u \in W^{1,p}(\Omega) \,.
\end{equation}
We say that $\Omega$ is a Lebesgue--Dirichlet $ W^{1,p}$-extension domain
if it admits a Lebesgue-Dirichlet $W^{1,p}(\Omega)$-extension operator.

 The $W^{1,p}$-extension operators constructed in \cite{Stein70} for
bounded Lipschitz domains are linear Lebesgue--Dirichlet $W^{1,p}$-
extension operators. Moreover, a half space admits a Lebesgue--Dirichlet
$W^{1,p}$-extension operator.
\begin{theorem}\label{thm1}
 Let $1<p<\infty$. Then every $W^{1,p}$-extension domain admits a linear
Lebesgue $W^{1,p}$-extension operator.

 However, there are no bounded Lebesgue--Dirichlet $W^{1,p}$-extension
domains.

 An unbounded domain $\Omega$ admits a Lebesgue--Dirichlet $ W^{1,p}
(\Omega)$-extension operator if and only if there is an extension operator
\begin{equation}\label{1.4}
 E: W^{1,p}(\Omega)\to W^{1,p}_{\loc}(\mathbb{R}^n)
 \end{equation}
 such that there is a constant $C$ with
 \begin{equation}\label{1.5}
 \| \nabla E(u)\|_{L^p(\mathbb{R}^n)}\leq C \|\nabla u\|_{L^p(\Omega)}
 \qquad\text{for all}\ u\in W^{1,p}(\Omega) \,.
 \end{equation}
\end{theorem}
 The non-existence of bounded Lebesgue-Dirichlet extension domains is
obvious: simply consider the function $u\equiv 1$. Then $\nabla u \equiv
0$ on $\Omega$ and a Lebesgue-Dirichlet extension operator would extend $u
$ to the function $Eu\equiv 1$ on $\R^n$.

For examples of unbounded domains with an extension operator satisfying
\eqref{1.4} and \eqref{1.5} we refer to the reader \cite{HDKP}.
\begin{remark}
One could, equivalently, consider domains of finite/infinite measure
instead of bounded/unbounded domains above. This is because an unbounded
extension domain necessarily is of infinite measure by the measure density
properties from Proposition \ref{HKT} and Lemma \ref{LemmaMCD} below. 
\end{remark}

The following theorem shows that we actually can extend from $\Omega$ to
any ball containing $\Omega$ with both Lebesgue and Dirichlet control.
\begin{theorem}\label{thm2} Let $1<p<\infty$. Let $\Omega$ be a bounded
$W^{1,p}$-extension domain and fix a ball $B$ with $\Omega \subset B$.
Then there is a linear operator  \begin{equation}
 E: W^{1,p}(\Omega)\to W^{1,p}(B)
 \end{equation}
 so that
\begin{equation}
\|\nabla E(u)\|_{L^p(B)}\leq C\|\nabla u\|_{L^p(\Omega)},
\end{equation}
\begin{equation}
 E(u)|_{\Omega}= u
\end{equation}
and also for all $1<q\leq \infty$,
\begin{equation}
 \|Eu\|_{L^q(B)}\leq C_q \| u\|_{L^q(\Omega)}, \hspace{2mm} \text
{when}\hspace{2mm} u\in L^{q}(\Omega).
\end{equation}
\end{theorem}
\subsection{First application: Heat Kernel Bounds}
We end this introduction by presenting two applications where the version
\eqref{eq:sobintext} of the Sobolev interpolation inequality with the $\theta$-geometric mean is crucial and where the version \eqref
{eq:sobintadditive} with the arithmetic mean is insufficient.

 The first one of these applications concerns a bound on the heat kernel
of the Neumann Laplacian on a $W^{1,2}$-extension domain.
\begin{theorem}
Let $\Omega\subset\R^n$ be a $W^{1,2}$-extension domain. Then the Neumann
heat semigroup $e^{t\Delta_\Omega^{\rm N}}$ with $t>0$, originally defined
on $L^2(\Omega)$, extends to a bounded operator $L^1(\Omega)\to L^\infty
(\Omega)$ satisfying
$$
\| e^{t\Delta_\Omega^{\rm N}} \|_{L^1(\Omega)\to L^\infty(\Omega)} \leq C
\max
\{ t^{-n/2},1\}
$$
with a constant $C$ depending only on $\Omega$.
\end{theorem}
The power $-n/2$ of $t$ is best possible for small $t$. As we will explain
in what follows, this is a consequence of having the optimal power of $\theta$ in Theorem \ref{sobint}.

This theorem is not new. It appears as Theorem 2.4.4 in \cite{Da}. The
theorem is stated there only for $t\leq 1$, but the bound for $t>1$
follows from a simple monotonicity argument; see Theorem 3.2.9 in \cite
{Da}. The proof of the theorem given in \cite{Da} for $n\geq 3$ is rather
direct
and is based on the Sobolev inequality with limiting exponent $2n/(n-2)$.
For $n\leq 2$ the proof given in \cite{Da} is not long, but somewhat
unnatural: The author observes that $\Omega\times\R^2\subset\R^{n+2}$ is a
$W^{1,2}$-extension domain and uses the Sobolev inequality on that set to
get a bound on the heat kernel in $\Omega\times\R^2$. This implies the
corresponding heat kernel bound on $\Omega$.

Our Theorem \ref{sobint} allows us to give a unified proof in all
dimensions. Indeed, in the special case $p=r=2$, $s=1$ we obtain the Nash
inequality $$\| u \|_{L^2(\Omega)}^{1+2/n} \leq C \| u \|_{W^{1,2}
(\Omega)} \| u \|_
{L^1(\Omega)}^{2/n} \,, $$ and, by Theorem 2.4.7 in \cite{Da}, this
implies the claimed heat kernel bound.

We emphasize that we do not claim that our proof is shorter than that given in \cite{Da}. (Indeed, it probably is not, given the moderately deep
facts -- Whitney decomposition, maximal theorem -- that enter our proof of
Theorem \ref{thm1}.) But we find it interesting to know that Nash's method works in any dimension.
In fact, since a heat kernel bound is equivalent to a Nash inequality, the
`unnatural proof' of the heat kernel inequality mentioned above already
implies the Nash inequality. This, however, is a rather convoluted proof
of the Nash inequality and motivated us to find a more direct one.

Finally, we emphasize that it is the case $s=1$ in Theorem \ref{sobint}
that is necessary to get heat kernel bounds. This is the case that is not
covered by our Theorem \ref{thm1}, but requires the truncation machinery
of \cite
{BCLS}, as explained in the proof of Theorem \ref{sobint}.
\subsection{Second application: Keller--Lieb--Thirring inequalities}
Let $1<p<\infty$ and let $\Omega\subset\R^n$ be a $W^{1,p}$-extension
domain. Given a nonnegative measurable function $W$ on $\Omega$, we
consider $$ E(W):= \inf_{0\not\equiv u\in W^{1,p}(\Omega)} \frac{\int_
\Omega (|\nabla u|^p - W |u|^p)\,dx}{\int_\Omega |u|^p\,dx} \,. $$ When
$p=2$, this is the ground state energy of the Schr\"odinger operator $-
\Delta -W$ in $L^2(\Omega)$ with Neumann boundary conditions.
In the case $\Omega=\R^n$ and $p=2$, Keller \cite{Keller61} asked about
the minimum value
of $E(W)$ among all $W$ with fixed $L^q$ norm and gave some heuristic
arguments showing that this should be related to a semilinear elliptic
equation.
Later, Lieb and Thirring \cite{LiebThirring1976} arrived independently at
the same minimization
problem and showed that it is equivalent to a Sobolev interpolation
inequality.
We show that a similar statement applies for arbitrary $1<p<\infty$ and
for arbitrary extension domains $\Omega$. Let
$$
S_{r}^\Omega := \inf_{0\neq u\in W^{1,p}(\Omega)} \frac{\|u\|_{W^{1,p}
(\Omega)}^\theta \|u\|_{L^p(\Omega)}^{1-\theta}}{\| u\|_{L^r(\Omega)}}
$$  with $\theta=\theta(r)$ given by $\frac1p - \frac{\theta}{n} = \frac1r
$. (This is
the relation \eqref{eq:asstheta} when $s=p$.)
\begin{theorem}\label{thm1.5}
Let $1<p<\infty$ and let $\Omega\subset\R^n$ be a $W^{1,p}$-extension domain. Let $\gamma>0$ if $n\geq p$ and $\gamma\geq 1-n/p$ if $n<p$. Then
$$ \inf_{0\neq W \in L^{\gamma+n/p}(\Omega)} \frac{E(W) + 1}{(\int_\Omega
W^{\gamma+n/p}\,dx)^{1/\gamma}} = -c_{n,p,\gamma} \left( S_r^\Omega
\right)^{-\frac{p(\gamma+n/p)}{\gamma}}. $$ Here $r$ is given by $$ \frac
{1}{\gamma+n/p} + \frac{p}{r} = 1 $$ and $$ c_{n,p,\gamma} := \frac{(n/(p
\gamma))^{n/(p\gamma)}}{(1+n/(p\gamma))^{1+n/(p\gamma)}} \,. $$  \end
{theorem}
The assumed conditions on $\gamma$ are equivalent to $r<np/(n-p)$ if $p
\leq n$ and
for $r\leq\infty$ if $p>n$, so, by Theorem \ref{sobint}, under these conditions we have
$S_r^\Omega>0$.

In particular, Theorem \ref{thm1.5} yields the lower bound
\begin{equation}
 \label{eq:kellerbound}
 E(W) \geq -1 - c_{n,p, \gamma} \left( S_r^\Omega \right)^{-
\frac{p(\gamma+n/p)}{\gamma}} \left(
 \int_\Omega W^{\gamma+n/p}\,dx \right)^{1/\gamma}.
\end{equation}
It is easy to construct, for each given $\gamma$ as in the theorem, a
sequence $(W_j)$
satisfying $\int_\Omega W_j^{\gamma+n/p}\,dx\to\infty$ and
$$
\limsup_{j
\to\infty} \frac{E(W_j)}{\left( \int_\Omega W_j^{\gamma+n/p}\,dx \right)^
{1/\gamma}} < \infty \,.
$$
(Indeed, we can take $W_j$ as an optimizer for the $\R^n$ problem,
truncated
at infinity and rescaled to a sufficiently small ball contained in $\Omega
$.) This means that the power $1/\gamma$ in the lower bound \eqref
{eq:kellerbound} on $E(W)$ in
terms of $\int_\Omega W^{\gamma+n/p}\,dx$ is best possible. Our proof of
Theorem \ref{thm1.5} shows that this power is a direct consequence of the
corresponding power $\theta$ in Theorem \ref{sobint} (with $s=p$). Had we
obtained the bound in Theorem \ref{sobint} with a larger power of $\theta
$, we would have gotten the lower bound on $E  (W)$ only with a power
larger than $1/\gamma$, which would not have been
optimal.

We provide the simple proof of Theorem \ref{thm1.5} in Section \ref
{sec:kellerliebthirring} below.

\section{Measure density condition}
The following result appeared in \cite[Proposition 1]{HKT}.
\begin{proposition}\label{HKT}
 Let $1\leq p < \infty$ and let $\Omega \subset \mathbb{R}^n$ be a $W^
{1,p}$-extension domain. Then there exists $\delta>0$ such that for all $x
\in \overline{\Omega}$ and $0 < r \leq 1$ one has $$
 |\Omega \cap B(x,r)|\geq \delta \, |B(x,r)| \,.
 $$
\end{proposition}
In our proof of Theorem \ref{thm1} we need the following variation of this
result.
 \begin{lemma}\label{LemmaMCD}
 Let $1\leq p<\infty$ and let $\Omega\subset\R^n$ be an unbounded
domain such that there is an extension operator
 $$ E: W^{1,p}(\Omega) \to W^{1,p}_{\loc}(\mathbb{R}^n)$$
 with
 \begin{equation}
 \label{edef} \|\nabla E(u)\|_{L^{p}(\mathbb{R}
^n)}\leq C\|\nabla u\|_{L^p(\Omega)}
 \qquad\text{for all}\ u\in W^{1,p}(\Omega) \,.
 \end{equation}
 Then there exists $\delta>0$ such that for all $x\in\overline\Omega$
and $r>0$ one has
 \begin{equation}\label{ahlfors}
 |\Omega \cap B(x,r)|\geq \delta \, |B(x,r)| \,.
 \end{equation}
\end{lemma}
\begin{proof}
 We divide the proof into three cases according to whether $p$ is
smaller, equal or larger than $n$.
 
 \emph{Case $1\leq p<n$.}
 
 We begin by proving that $|\Omega|=\infty$. Suppose to the contrary
that $|\Omega| <\infty$. \\
 Since $r\mapsto |\Omega\setminus B(0,r)|$ is continuous and
nonincreasing, for given $j\in \mathbb{N}$ we can pick $r_j$ such that 
\begin{equation}\label{omega}
 |\Omega \setminus B(0,r_j)| =2^{-j}|\Omega|.
 \end{equation}
 Fix $j$ and define
 \begin{equation}\label{mueq}
 u(x):= \begin{cases}
 1-\frac{d(y,B(x,r_j))}{r_{j+1}-r_{j}} & \text{if}\ y\in B
(0,r_{j+1})\cap\Omega \,,\\ 0 & \text{if}\ y\in \Omega
\setminus B(0,r_{j+1}) \,.
 \end{cases}
 \end{equation}
 Then $u$ is $\frac{1}{r_{j+1}-r_j}$-Lipschitz.
 
 By using the assumption \eqref{edef} on $E$ and \eqref{mueq} we get
that
 \begin{equation}\label{nex}
 \|\nabla E(u)\|_{L^{p}(\R^n)}\leq C\frac{1}{r_{j+1}-r_{j}}|[B
(0,r_{j+1})\cap \Omega] \setminus B(0,r_{j})|^{\frac{1}{p}}.
 \end{equation}
 We shall make use of the Sobolev--Poincar\'e inequality (see, e.g.,
\cite[Theorem 8.12]{LiebLoss01})
 \begin{equation}\label{sobpq}
 \left(\int_{B(0,r_{j+2})}|E(u)(x)-(E(u))_{B(0,r_{j+2})}|^{\frac
{np}{n-p}}dx\right)^{\frac{n-p}{np}}\leq C\left(\int_{B(0,r_{j+2})}|\nabla
E(u)(x)|^p dx\right)^{\frac{1}{p}}. \end{equation}
 The integrand on the left side is $\geq |1-a|$ on $B(0,r_j)\cap\Omega
$ and $\geq |a|$ on $[B(0,r_{j+2})\cap \Omega]\setminus B(0,r_{j+1})$,
where $a=(E(u))_{B(0,r_{j+2})}$. Since $\max\{|1-a|,|a|\}\geq 1/2$, by
combining \eqref{mueq} and \eqref{sobpq}, we conclude that \begin{equation}
 \begin{split} & \left( \min\{|B(0,r_j)\cap\Omega|, |
[B(0,r_{j+2})\cap \Omega]\setminus B(0,r_{j+1})|\}\right)^{\frac{n-p}
{np}}\\
& \quad \leq C\frac{1}{r_{j+1}-r_{j}}|[B(0,r_{j+1})\cap
\Omega] \setminus B(0,r_{j})|^{\frac{1}{p}}.
 \end{split}
 \end{equation}
 By recalling the construction of $r_j$ in \eqref{omega}, we can
express all measures that appear in this inequality in terms of $j$ and we
arrive at the bound \begin{equation}
 r_{j+1}-r_{j} \leq C \frac{(2^{-j}|\Omega| -2^{-(j+1)}|\Omega|)^
{\frac{1}{p}}}{(2^{-(j+1)}|\Omega| -2^{-(j+2)}|\Omega|)^{\frac{n-p}{np}}}
\,, \end{equation}
 which implies that
 \begin{equation}\label{sumew}
 r_{j+1}-r_{j} \leq C |\Omega|^{\frac{1}{n}}2^{-\frac{j}{n}}.
 \end{equation}
 This contradicts $\sum_{j=1}^{\infty}(r_{j+1}-r_{j}) = \lim_{j\to
\infty} r_j - r_1 = \infty$ and therefore proves that $|\Omega|=\infty$.
 
 We now turn to the proof of the measure density bound. We fix $x\in
\overline{\Omega}$ and $r>0$ and note that $|B(x,r)\cap \Omega|>0$. Let
$r_0:=r$ and choose a decreasing sequence of $(r_j)$ such that
 \begin{equation}\label{seq}
 |B(x,r_j)\cap\Omega|= \frac{1}{2^j}|B(x,r)\cap \Omega|.
 \end{equation}
 Since $|\Omega|=\infty,$ we may further choose $r_{-1}>r$ so that
 \begin{equation}\label{seq1}
 |B(x,r_{-1})\cap\Omega|= 2|B(x,r)\cap \Omega|.
 \end{equation}
 Fix $j$ and define
 \begin{equation}\label{udef}
 u(y):= \begin{cases}
 1- \frac{1}{r_{j}-r_{j+1}} \, d(y,B(x,r_{j+1})) & \text{if}\ y \in B(x,r_j)\cap \Omega \,,\\
 0 &\text{if}\ y\in \Omega\setminus B(x,r_j) \,.
 \end{cases}
 \end{equation}
 Then $u$ is $\frac{1}{r_j-r_{j+1}}$-Lipschitz.
 
 By using the assumption \eqref{edef} on $E$ and \eqref{udef} we get
that
 \begin{equation}\label{nex2}
 \|\nabla E(u)\|_{L^{p}(\R^n)}\leq C \frac{|[B(x,r_{j})\cap
\Omega ]\setminus B(x,r_{j+1})|^{\frac{1}{p}}}{r_j-r_{j+1}}.
 \end{equation}
 Once again we use the Sobolev--Poincar\'e inequality \eqref{sobpq},
this time for $j\ge 0$ in the form
 \begin{equation}
 \label{eq:sobpq2} \left(\int_{B(x,r_{j-1})}|E
(u)(y)-(E(u))_{B(x,r_{j-1})}|^{\frac{np}{n-p}}dy \right)^{\frac{n-p}{np}}
\leq C\left(\int_{B(x,r_{j-1})}|\nabla E(u)(y)|^p dy \right)^{\frac{1}
{p}}. \end{equation}
 The integrand on the left side is $\geq |1-a|^{np/(n-p)}$ on $B(x,r_
{j+1})\cap\Omega$ and $\geq |a|^{np/(n-p)}$ on $[B(x,r_{j-1})\cap
\Omega]\setminus B(x,r_j)$, where $a := (E(u))_{B(x,r_{j-1})}$. Since $
\max\{|1-a|,|a|\}\geq 1/2$, by combining \eqref{nex2} and \eqref
{eq:sobpq2}, we conclude that $$
 ( \min\{ | B(x,r_{j+1})\cap\Omega|, |[B(x,r_{j-1})\cap
\Omega]\setminus B(x,r_j)|\})^\frac{n-p}{np} \leq C\frac{| [B(x,r_{j})\cap \Omega ]\setminus B(x,r_{j+1})|^{\frac{1}{p}}}{r_j-r_{j+1}}. $$
 According to \eqref{seq1}, we have
 $$
 |[B(x,r_{j-1})\cap\Omega]\setminus B(x,r_j)| = |B(x,r_{j-1})\cap
\Omega| - |B(x,r_j)\cap\Omega| = \frac1{2^j} |B(x,r)\cap\Omega| \,,
 $$
 so the previous inequality implies
 $$
 r_j - r_{j+1} \leq C \left( \frac{1}{2^j} |B(x,r)\cap\Omega| \right)^
\frac1n.
 $$ As $r_j\to 0$ for $j\to\infty$, we obtain
 $$
 r=r_0 \leq \sum_{j=0}^{\infty}(r_{j}-r_{j+1})
 \leq C \sum_{j=0}^\infty \left( \frac{1}{2^j} |B(x,r)\cap\Omega|
\right)^\frac{1}{n}.
 $$
 Hence
 \begin{equation}\label{eqR}
 r^n\leq C |\Omega\cap B(x,r)|,
 \end{equation}
 which is the claimed inequality.
 
 \medskip
 
 \emph{Case $p=n$.}
 
 We fix $x\in\overline\Omega$ and $r>0$ and abbreviate $B := B(x,r)$.
Define
 \begin{equation}\label{udefi}
 u(y) := \begin{cases}
 1 &\ {\rm if}\ y\in \frac{1}{3}B \cap \Omega,\\
 0 &\ {\rm if}\ y \in \Omega \setminus \frac{2}{3}B,\\
 2- \frac{3 |y-x|}{r} &\ {\rm if}\ y \in (\frac{2}{3} B
\setminus \frac{1}{3}B) \cap \Omega.
 \end{cases}
 \end{equation}
 Then $u$ is $ \frac{3}{r}$-Lipschitz.\\
 
 By our assumption \eqref{edef} on $E$ and \eqref{udefi} there exists
$C>0$ such that
 \begin{equation}\label{jkl} \|\nabla E(u)\|_{L^{n}(\R^n)}^n
\leq C \|\nabla u\|^{n}_{L^{n}(\Omega)}\leq C \frac{3^n}{r^n}|[\tfrac{2}
{3} B \setminus \tfrac{1}{3}B]\cap \Omega| \leq 3^n C \frac{|B\cap
\Omega|}{r^n}. 
\end{equation}
 
 Meanwhile, we claim that there exists $C>0$ such that
 \begin{equation}\label{jkl1}
 \int_{10B}|\nabla E(u)|^n dx \geq C.
 \end{equation}
 To prove this, we proceed as in the proof of \cite[Proposition 13]
{HaKoTu08}. We recall that the Hausdorff 1-content of a set $X\subset\R^n$
is $$
 \mathcal H_\infty^{(1)}(X) = \inf \sum_i r_i \,,
 $$ 
 where the infimum is taken over all countable covers of $X$ by
balls and $r_i$ is the radius of the $i$-th ball in this cover. Let $E:=
\frac{1}{3}B\cap \Omega$ and $F:= (B \setminus\frac{2}{3}B)\cap \Omega$.
Then, as shown in the proof of \cite[Proposition 13]{HaKoTu08}, we have $\min\{\mathcal{H}_{\infty}^{1}(E), \mathcal{H}_{\infty}^{1}(F)\} \geq
\frac{r}{3}$. Hence by \cite[Theorem 5.9]{HK} (and its slight extension in
\cite[Lemma 15]{HaKoTu08}), there exists $C>0$ such that \eqref{jkl1}
holds.
 
 Combining \eqref{jkl} and \eqref{jkl1}, we find
 \begin{equation}
 \frac{|B\cap \Omega|}{r^n} \geq C,
 \end{equation}
 which is the claimed inequality.
 
 \medskip
 
 \emph{Case $p>n$.}
 
 We fix $x\in \overline{\Omega}$ and $r>0$ and abbreviate $B:=B(x,r)$.
Let
 \begin{equation}\label{udefi2}
 u(y) := \begin{cases}
 1- \frac{|x-y|}{r} &\ {\rm if}\ y\in B\cap \Omega,\\
 0 &\ {\rm if}\ x\in \Omega\setminus B.
 \end{cases}
 \end{equation}
 Then $u $ is $\frac{1}{r}$-Lipschitz.

 By our assumption \eqref{edef} on $E$ and \eqref{udefi2} there
exists $C>0$ such that
 \begin{equation}\label{exs}
 \|\nabla E(u)\|_{L^{p}(\R^n)}\leq C \frac{|B\cap \Omega|^{\frac
{1}{p}}}{r}.
 \end{equation}
 By Morrey's inequality (see, e.g., \cite[Remark 12.49]{Leoni17}), we
have, for all $z,y\in\R^n$,
 \begin{equation}\label{morrey}
 |E(u)(z)-E(u)(y)|\leq C|z-y|^{1-\frac{n}{p}} \|\nabla E(u)\|_{L^
{p}(\R^n)} \,.
 \end{equation}
 Choosing $y \in \overline{B}\subset \Omega$ such that $u(y)=0$ and $z
\in B(x,\frac{r}{2})\cap \Omega$ we obtain
 \begin{equation}
 \label{eq:morrey2}
 \frac12 \leq C r^{1-\frac np} \|\nabla E(u)\|_{L^{p}(\R^n)}
\,.
 \end{equation}
 
 Combining \eqref{exs} and \eqref{eq:morrey2}, we get that
 \begin{equation}
 \frac{1}{2} \leq C \cdot r^{1-\frac{n}{p}} \frac{|B\cap
\Omega|^{\frac{1}{p}}}{r}
 \end{equation}
 Hence
 \begin{equation}
 |B\cap \Omega|\geq Cr^n \geq C^{'}|B|,
 \end{equation}
 which is the claimed inequality. This concludes the proof of Lemma
\ref{LemmaMCD}. \end{proof}

\section{Extension operators}
Our goal in this section is to prove Theorems \ref{thm1} and \ref{thm2}
concerning the Lebesgue and the Lebesgue--Dirichlet extension property.
\subsection{Whitney decomposition}
Our proof of Theorem \ref{thm1} will rely on the Whitney decomposition. We
begin with a reminder of this decomposition; see, e.g., \cite{Stein70}.
\begin{lemma}\label{Lemma 1}
 Let $U\subsetneq\R^n$ be an open set. Then there exists a collection
of closed cubes $\{Q_i\}_{i=1}^{\infty}$ so that
 \begin{enumerate}
 \item $\bigcup_{i}Q_i= U$,
 \item Then interior of $Q_i$ are mutually disjoint,
 \item if $x\in 5Q_i$, then $2\diam{Q_i}< \dist(x,\mathbb{R}^n
\setminus U)< 8\diam{Q_i}$,
 \item For each $Q_i$ there is $x_i^{\star}\in \mathbb{R}^n
\setminus U$
 such that $\dist(x_i^{\star}, Q_i)<15\diam{Q_i}$,
 \item There is $M \in \mathbb{N}$ such that $\sum_{i=1}^
{\infty}\chi_{15Q_i}(x)\leq M$ for all $x\in U$. \end{enumerate}
\end{lemma}
We refer the above collection $\{Q_i\}_{i=1}^{\infty}$ as a Whitney
covering of $U$. With such a Whitney covering of $U$, we can associate a
Lipschitz partition of unity.
\begin{lemma}\label{Lemma 2}
 Given a Whitney covering $\{Q_i\}_{i=1}^{\infty}$ of $U$, there
exists a partition of unity $\{\phi_{i}\}_{i=1}^{\infty}$ so that 
\begin{enumerate}
 \item $\supp\phi_{i}\subset 2Q_i$,
 \item $\phi_{i}(x)\leq M^{-1}$ for all $x\in Q_i$,
 \item there is a constant $K$ such that each $\phi_{i}$ is $K
(\diam{Q_i})^{-1}$-Lipschitz,
 \item $\sum_{i=1}^{\infty}\phi_{i}(x)= \chi_{U}(x)$ for all $x
\in\R^n$.
 \end{enumerate}
\end{lemma}
\subsection{Proof of Theorem \ref{thm1}. First part.}
Let $\Omega$ be a $W^{1,p}$-extension domain. We are going to construct a
linear Lebesgue $W^{1,p}$-extension operator.
 Let $\mathcal{W}= \{Q_i\}_{i\in I}$ be a Whitney covering of the open set
$U = \R^n\setminus \overline\Omega$ as in Lemma \ref{Lemma 1} and let $
\{\phi_i\}_{i\in I}$ be the Lipschitz partition of unity as in Lemma \ref
{Lemma 2}. We abbreviate $$
r_i := \diam{Q}_i \,.
$$

Let
$$
J:=\{ i\in I :\ r_i\leq 1\}
\qquad\text{and}\qquad
\mathcal{Q}=\{Q_i\}_{i\in J} \,.
$$
Next, let $x_i^\star \in\overline\Omega$ be as in Lemma \ref{Lemma 1} and
define
$$
B_{i}^\star := B(x_i^\star,r_i)\cap \Omega.
$$
Then
$$
B_i^\star \subset 20Q_i.
$$
It follows from Proposition \ref{HKT} that there exists a $C>0$ such that
$$
|Q_i|\leq C |B_{i}^\star|.
$$
Let
$$
V:= \{x\in \mathbb{R}^n: \dist(x,\Omega)\leq 2\}
$$
and for $x\in V\setminus\overline{\Omega}$ let
$$
I_x := \{ i\in I :\ x\in 2 Q_i \} \,.
$$
Then from Lemma \ref{Lemma 1}, the number of elements in $I_x$ is bounded
by $M$. Also, we claim that
$$
I_x \subset J \,.
$$ Indeed, if $i\in I\setminus J$, then $r_i> 1$ and hence for $z\in 2Q_i,
\dist(z,\Omega)\geq 2r_i> 2$, so $2Q_i\cap V=\emptyset$. Hence, $I_x
\subset J$, as claimed.
As a consequence,
$$
\sum_{i\in I_x}\phi_i(x)= \sum_{i\in I}\phi_i(x)=\sum_{i\in J}\phi_i(x)=1$$

for $x\in V\setminus \overline{\Omega}$.

Take a Lipschitz cutoff function $\Phi$ such that $\Phi(x)=0$, when $d(x,
\partial \Omega) >2$, $\Phi\equiv 1$ on $\Omega$ and $|\nabla \Phi(x)|\leq
\frac{1}{2}$ for $x\in V$.

For $u\in C_0^{\infty}(\mathbb{R}^n)$, we define $E_1$ by setting
\begin{equation}\label{def1}
E_1u(x):= \begin{cases}
 \left( \sum_{i\in J}(u)_{B_i^\star}\phi_{i}(x)\right)\Phi(x) &\ {\rm if}\
x\in \mathbb{R}^n\setminus\overline{\Omega},\\
 u(x) &\ {\rm if}\ x\in \Omega.
 \end{cases}
\end{equation}
Then $E_1$ is linear.\\

Let $\widetilde{\mathcal{W}}:= \{Q_i\in \mathcal{W}: Q_i\cap V \neq
\emptyset\}$. For $Q_0\in\widetilde{\mathcal{W}}$, we have  \begin
{equation}\label{eq1.}
\begin{split} \|\nabla E_1 u\|_{L^p(Q_0)}^p&= \|\nabla \left( \sum_{i\in
J}(u)_{B_i^\star}\phi_{i}\right)\Phi\|_{L^p(Q_0)}^p\\
&\leq \|\left(\nabla \sum_{i\in J}(u)_{B_i^\star}\phi_{i}\right)\Phi\|_
{L^p(Q_0)}^p+ \|\left(\sum_{i\in J}(Eu)_{B_i^\star}\phi_{i}\right)\nabla
\Phi\|_{L^p(Q_0)}^p\\ &\leq C\left(\|\nabla \sum_{i\in J}(u)_{B_i^
\star}\phi_{i}\|_{L^p(Q_0)}^p+ \|\sum_{i\in J}(u)_{B_i^\star}\phi_{i}\|_
{L^p(Q_0)}^p\right)\\ &\leq C\left(\| \sum_{i\in J}(u)_{B_i^\star} -(Eu)_
{B_0^\star})\nabla \phi_i\|_{L^p(Q_0)}^p +\|\sum_{i\in J}(u)_{B_i^
\star}\phi_{i}\|_{L^p(Q_0)}^p\right) \end{split}
\end{equation}
For fixed $i\in \mathbb{N}$,
\begin{equation}\label{ai.}
\begin{split}
|(Eu)_{B_{i}^\star}-(Eu)_{B_0^\star}|
 &=|\barint_{B_i^\star}Eu(x)dx-\barint_{B_0^\star}Eu(y)dy|\\
 &\leq \barint_{B_i^\star}\barint_{B_0^\star}|Eu(x)-Eu(y)|dxdy.
\end{split}
\end{equation}
Since $Eu \in W^{1,p}(\mathbb{R}^n)\cap C(\Omega)$, for $x,y\in \Omega$,
we get
\begin{equation}\label{maximal.}  |Eu(x)-Eu(y)|\leq C|x-y|(M(|\nabla Eu|)
(x)+ M(|\nabla Eu|)(y))
\end{equation}
Using \eqref{ai.}, \eqref{maximal.} and $(5)$ of Lemma \ref{Lemma 1}, we
get that
\begin{eqnarray}\label{eq2.}
 |(Eu)_{B_{i}^\star}-(Eu)_{B_0^\star}|&\leq& \barint_{B_i^\star}\barint_
{B_0^\star}C|x-y|(M(|\nabla Eu|)(x)+ M(|\nabla Eu|)(y))dxdy\nonumber\\  &
\leq & C\diam (Q_0)\Bigg[\barint_{B_i^\star}M(|\nabla Eu|)(x)dx+\barint_
{B_0^\star}M(|\nabla Eu|)(y)dy\Bigg].
\end{eqnarray}
Hence by using \eqref{eq2.}, Proposition \ref{HKT} and Lemma \ref{Lemma
1}, we conclude that
\begin{equation}\label{eq1..}
\begin{split}
 \| \sum_{i\in J}(Eu)_{B_i^\star} -(Eu)_{B_0^\star})\nabla \phi_i\|_{L^p
(Q_0)}^p
 &\leq \sum_{2Q_i\cap Q_0\neq \phi}\frac{C}{(\diam (Q_i))^p}\int_{Q_0} |
(Eu)_{B_{i}^\star}-(Eu)_{B_0^\star}|^pdx\\
 &\leq C \int_{Q_0} |M(M(|\nabla Eu|))(x)|^pdx.
\end{split}
\end{equation}
Combining \eqref{eq1.} and \eqref{eq1..}, we conclude that
\begin{equation}  \|\nabla E_1 u\|_{L^p(Q_0)}^p \leq C\int_{Q_0} |M(M(|
\nabla Eu|))(x)|^pdx.
\end{equation}
Write $v$ for the zero extension of $u$ to the complement of $\Omega.$
Now for $x\in Q_0,$
\begin{equation}\label{mx}
\begin{split}
 |\sum_{i\in J}(Eu)_{B_i^\star}\phi_i(x)|
 &\leq \sum_{2Q_i\cap Q_0\neq \phi}\barint_{B_i^\star}|Eu(y)|dy\\
 &\leq C \sum_{2Q_i\cap Q_0\neq \phi} \frac{|20 Q_0|}{|B_i^\star|} \barint
_{20Q_0} |v(y)|dy\\
 &\leq C \barint _{20Q_0} |v(y)|dy\\
&\leq C M(v)(x).
\end{split}
\end{equation}
Hence,
\begin{eqnarray}
\|\sum_{i\in J}(Eu)_{Q_i^\star}\phi_{i}\|_{L^p(Q_0)}^p&\leq& C\int_{Q_0}|M
(v)(x)|^pdx.
\end{eqnarray}
By summing over the cubes $Q_0$, we conclude that
\begin{equation}
\begin{split}
\int_{\mathbb{R}^n\setminus \overline{\Omega}} |\nabla E_1 u(x)|^pdx +
\int_{\mathbb{R}^n\setminus \overline{\Omega}} | E_1 u(x)|^pdx
 \leq C &\int_{\mathbb{R}^n}|M(v)(x)|^pdx \\
 &+\int_ {\mathbb{R}^n}|M(M(|\nabla Eu|))(x)|^p dx.
 \end{split}
\end{equation} Since the maximal operator is bounded on $L^p(\mathbb{R}
^n)$, when $p>1$, we get that
\begin{equation}\label{maxi}
\begin{split}
 \int_{\mathbb{R}^n\setminus \overline{\Omega}} |\nabla E_1 u(x)|^pdx +
\int_{\mathbb{R}^n\setminus \overline{\Omega}} |E_1 u(x)|^pdx \leq
C & \int_{\mathbb{R}^n}|v(x)|^pdx\\
&+\int_ {\mathbb{R}^n}|\nabla Eu(x)|^p dx.
\end{split}
\end{equation}

Combining the definition of $E$, \eqref{maxi} and $|\partial \Omega|=0$,
we conclude that
\begin{equation}\label{4.12}
\begin{split}
 \int_{\mathbb{R}^n\setminus {\Omega}} |\nabla E_1 u(x)|^pdx +\int_
{\mathbb{R}^n\setminus {\Omega}} |\nabla E_1 u(x)|^pdx \leq C&\int_
{\Omega}|u(x)|^pdx\\  &+\int_{\Omega}|\nabla u(x)|^pdx.
 \end{split}
\end{equation} Moreover, since $u\in C_0^{\infty}(\mathbb{R}^n)$, we have
$|u(x)-u(y)|\leq L|x-y|$ for some $L$ whenever $x,y\in \Omega$. This gives
us Lipschitz estimate together with \eqref{def1} implies that $E_1u$ is
Lipschitz in $\mathbb{R}^n\setminus \partial \Omega$. Hence we get weak
differentiability and it follows from \eqref{4.12} that $E_1u \in W^{1,p}
(\mathbb{R}^n)$.
Since $\Omega$ is a $W^{1,p}$-extension domain and $C_0^{\infty}(\mathbb
{R}^n)\cap W^{1,p}(\Omega)$ is dense in $W^{1,p}(\Omega)$, we conclude
that $E_1$ is a $W^{1,p}$-extension operator. From \eqref{mx} together
with the definition \eqref{def1}, we get that $E_1$ is a linear Lebesgue
$W^{1,p}$-extension operator.
\subsection{Proof of Theorem \ref{thm1}. Second part.}
Towards the second claim, let $\Omega$ be a bounded Lebesgue-Dirichlet $W^
{1,p}$-extension domain. Let $E$
be a bounded $W^{1,p}$-extension operator such that
\begin{equation}\label{ideq1}  \|\nabla E(u)\|_{L^p(\mathbb{R}^n)}\leq C \|\nabla u\|_{L^p(\Omega)}
\end{equation}
and
\begin{equation}\label{ideq}
 \| E(u)\|_{L^p(\mathbb{R}^n)}\leq C \| Eu\|_ {L^p(\Omega)}
\end{equation}
for all $u\in W^{1,p}(\Omega)$.
Take $u(x)= \chi_{\Omega}(x)$. From \eqref{ideq1}, we have $\|\nabla E
(u)\|_
{L^{p}(\mathbb{R}^n)}=0$, which implies that $E(u)$ is a constant
function.
Since $E(u)\in W^{1,p}(\mathbb{R}^n)$, $E(u)(x)$ will be $\chi_{\mathbb{R}
^n}
(x)$, which contradicts \eqref{ideq}.
\subsection{Proof of Theorem \ref{thm1}. Third part.}
Now assume that $\Omega$ is unbounded and let $E_1$ be a Lebesgue--
Dirichlet $W^{1,p}(\Omega)$-extension operator. Then $E_1$ is a  bounded
$W^{1,p}(\Omega)$-extension operator and there exists a $C>0$ such that
\begin{equation}\label{meq1}
 \|\nabla E_1(u)\|_{L^p(\mathbb{R}^n)}\leq C \|\nabla u\|_{L^p(\Omega)}
 \end{equation}
 and
 \begin{equation}
 \| E_1(u)\|_{L^p(\mathbb{R}^n)}\leq C \|u\|_{L^p(\Omega)},
 \end{equation}  for all $u\in W^{1,p}(\Omega)$. Define $E_2=E_1$,
$$
 E_2: W^{1,p}(\Omega)\to W^{1,p}_{\loc}(\mathbb{R}^n) \,.
$$
 By using \eqref{meq1}, we get that
 \begin{equation}
 \|\nabla E_2(u)\|_{L^{p}(\mathbb{R}^n)}=\|\nabla E_1(u)\|_{L^{p}
(\mathbb{R}^n)} \leq C\|\nabla u\|_{L^p(\Omega)},
 \end{equation}
 which is the claimed bound.

Conversely, let $\Omega$ be unbounded and assume there is an extension
operator
\begin{equation}
 E_2: W^{1,p}(\Omega) \to W^{1,p}_{\loc}(\mathbb{R}^n)
\end{equation}
satisfying, for some constant $C$,
 \begin{equation}\label{homext}
 \|\nabla E_2(u)\|_{L^p(\mathbb{R}^n)}\leq C \|\nabla u\|_{L^p(\Omega)}
 \qquad\text{for all}\ u\in W^{1,p}(\Omega) \,.
 \end{equation}
Let $\mathcal{Q}=\{Q_i\}_{i\in I}$ be a Whitney covering of $\mathbb{R}^n
\setminus \overline{\Omega}$ and let $\{\phi_i\}_{i\in I}$ be the
associated Lipschitz partition of unity. For each $i\in I$, let $x_i^\star
$ be as in Lemma \ref{Lemma 1}. Set $r_i=\diam{Q_i}$. We define $$B_{i}^
\star= B(x_{i^\star},r_i)\cap \Omega.$$ Then $$B_{i}^\star \subset20 Q_i.
$$ From Lemma \ref{LemmaMCD}, there exist $C>0$, such that $$|Q_i|\leq C |
B_i^\star|.$$
Let $u\in W^{1,p}(\Omega)\cap C_0^{\infty}(\mathbb{R}^n)$ and $v$ be the zero extension of $u$. Define $E_1$ by setting
\begin{equation}\label{def2} E_1u(x):= \begin{cases}
 \sum_{i\in I}(E_2u)_{B_i^\star}\phi_{i}(x), &\ {\rm if}\ x\in \mathbb{R}
^n\setminus\overline{\Omega},\\
 u(x),&\ {\rm if}\ x\in \Omega.
 \end{cases}
\end{equation}
Then $E_1$ is a linear map. \\
We continue towards $\|\nabla E_1 u\|_{L^p(B)}\leq C\|\nabla u\|_{L^p
(\Omega)}$.\\
For $Q_0\in\mathcal Q$, we have
\begin{equation}\label{eq1*}
\begin{split}
\|\nabla E_1 u\|_{L^p(Q_0)}^p&=\|\nabla \sum_{i=1}^\infty (E_2u)_{B_i^
\star}\phi_i\|_{L^p(Q_0)}^p\\
&\leq \|\nabla (\sum_{i=1}^\infty((E_2u)_{B_i^\star} -(E_2u)_{B_0^
\star}+(E_2u)_{B_0^\star})\phi_i)\|_{L^p(Q_0)}^p\\
&\leq \| \sum_{i=1}^\infty((E_2u)_{B_i^\star} -(E_2u)_{B_0^\star})\nabla
\phi_i\|_{L^p(Q_0)}^p\\ &\leq \|\sum_{2Q_i\cap Q_0\neq \phi}((E_2u)_{B_i^
\star} -(E_2u)_{B_0^
\star})\nabla \phi_i\|_{L^p(Q_0)}^p.
\end{split}
\end{equation}
Precisely as in the proof of Theorem \ref{thm1}, we get that
\begin{equation}\label{eq5.10} \|\nabla E_1u\|^p_{L^p(Q_0)} \leq C \int_
{Q_0}|M(M(|\nabla E_2u|))(x)|^p dx
\end{equation}
By summing over $Q_0$'s we conclude that
\begin{equation}\label{eq5.11}
 \int_{\mathbb{R}^n\setminus \overline{\Omega}}|\nabla E_1u(x)|^pdx \leq C
\int_{\mathbb{R}^n}|M(M(|\nabla E_2u|))(x)|^p dx.
\end{equation}
By using the boundedness of the maximal operator on $L^p(\mathbb{R}^n)$,
when $p>1$ and \eqref{homext}, from \eqref{eq5.11}, we get that
\begin{equation}
 \int_{\mathbb{R}^n\setminus \overline{\Omega}}|\nabla E_1u(x)|^pdx \leq C
\int_{\mathbb{R}^n}|\nabla E_2u(x)|^pdx\leq C \int_{\Omega}|\nabla u(x)|
^pdx. \end{equation}
Let $x\in \mathbb{R}^n \setminus \overline{\Omega}$. Pick $Q_0\in \mathcal
{Q}$ such that $x\in Q_0$. Then
\begin{equation}\label{eq5.13}
\begin{split}  |E_1u(x)| &=|\sum_{i\in I}(E_2u)_{B_i^\star}\phi_{i}(x)|\\
 &\leq \sum_{2Q_i\cap Q_{0}\neq \phi}\barint_{B_i^\star}|E_2u(y)|dy\\
 &\leq \max_{2Q_i \cap Q_{0}\neq \phi} \barint_{B_i^\star}|u(y)|dy\\
 &\leq CM(v)(x). \end{split}
\end{equation}
By summing over the cubes $Q_0$, we conclude that
\begin{equation}\label{maxi1}
\int_{\mathbb{R}^n\setminus \overline{\Omega}} |E_1 u(x)|^pdx
\leq C \int_{\mathbb{R}^n}|M(v)(x)|^pdx.
\end{equation}
Since the maximal operator is bounded on $L^p(\mathbb{R}^n)$ when $p>1$,
and $|\partial\Omega|=0$, we conclude from \eqref{eq5.10} and \eqref
{maxi1} that \begin{equation}
 \int_{\mathbb{R}^n\setminus \Omega} | E_1 u(x)|^pdx + \int_{\mathbb{R}^n
\setminus \Omega} |\nabla E_1 u(x)|^pdx
\leq C \int_{\Omega}|u(x)|^pdx + \int_{\mathbb{R}^n\setminus \Omega} |
\nabla u(x)|^pdx.
\end{equation}
By \eqref{eq5.13} and the density of $C_0^{\infty}(\mathbb{R}^n)\cap W^
{1,p}(\Omega)$ in $W^  {1,p}(\Omega)$, we conclude that $E_1$ is a linear
Lebesgue-Dirichlet $W^{1,p}$-extension operator.
\subsection{Proof of Theorem \ref{thm2}}
Let $1<p<\infty$ and let $\Omega\subset\R^n$ be a bounded $W^{1,p}$-
extension domain and fix a ball $B\subset\mathbb{R}^n$ such that $\Omega
\subset \subset B$.

Let $\mathcal{W}_1$ be a Whitney covering of $\mathbb{R}^n\setminus
\overline{\Omega}$ and $\{\phi_i\}_{i\in I}$ be the associated Lipschitz
partition of unity as in Lemmas \ref{Lemma 1} and \ref{Lemma 2}.

Define
\begin{equation*}
 \mathcal{W}_2:= \{Q_i\in \mathcal{W}_1 \hspace{2mm} \text{such
that}\hspace{2mm} Q_i\cap B\neq \emptyset\}.
\end{equation*}
For each $Q_i\in \mathcal{W}_2$, define
\begin{equation}
 B_i^\star:= B(x_i^\star, \diam(Q_i))\cap \Omega.
\end{equation}
Since $\Omega$ is a bounded $W^{1, p}$-extension domain, for every $u\in
W^{1, p}(\Omega)$, there exists a function $E_1(u)\in W^{1, p}(B)$ such
that $E_1(u)\big|_{\Omega}\equiv u$ and
\begin{equation} \label{eq:extension1}
 \|E_1(u)\|_{W^{1, p}(B)}\leq C\|u\|_{W^{1,p}(\Omega)},
\end{equation}
where $C>0.$\\
 Define
\begin{equation}
 E_2(u):= E_1(u-u_\Omega)+u_\Omega.
\end{equation}
Since $\Omega$ is a bounded $W^{1,p}$-extension domain, by using \cite
{HDKP}, we get that
\begin{equation}\label{eq:extension2}
\|\nabla E_2(u)\|_{L^{ p}(B)}\leq C\|\nabla u\|_{L^{p}(\Omega)}, \end
{equation}
where $C>0$.\\
Define $a_i$ for each $Q_i\in \mathcal{W}_2$ by setting
\begin{equation}
 a_i:= \begin{cases}
 (u)_{B_i^\star}, &\ {\rm if}\ \diam(Q_i)\leq 1,\\
 (u)_\Omega, &\ {\rm if}\ \diam(Q_i)>1.
 \end{cases}
\end{equation}
Define $E_3$ by setting
\begin{equation}\label{def}
E_3u(x):= \begin{cases}
 \sum_{i=1}^{\infty}a_i\phi_{i}(x), &\ {\rm if}\ x\in B\setminus\overline
{\Omega},\\
 u(x),&\ {\rm if}\ x\in \Omega.
\end{cases}
\end{equation}
The desired estimates follow as in the proof of the third part of Theorem
\ref{thm1}, using
\eqref{eq:extension2} instead of \eqref{homext} and Proposition \ref{HKT}
instead of Lemma \ref{LemmaMCD}. Notice that there are only finitely many
$Q_i$ with $\diam(Q_i)>1$ in $\mathcal{W}_2$.

\section{Proof of the Sobolev interpolation inequality}
In this section we prove our main result, a Sobolev interpolation
inequality with the $\theta$-geometric mean on extension domains.
\begin{proof}[Proof of Theorem \ref{sobint}]
 Let $p,r,s,\theta$ be as in the theorem, but assume initially that $s>1$.
Let $u\in W^{1,p}(\Omega)\cap L^s(\Omega)$ and let $Eu$ be its extension.
The key observation is that we have not only
$$
\| \nabla E u \|_{L^{p}(\R^n)} \leq C \| u \|_{W^{1,p}(\Omega)} \,,
$$
but also
$$\| Eu \|_{L^s(\R^n)} \leq \| u \|_{L^s(\Omega)} \,.
$$
This is the content of Theorem \ref{thm1}. Therefore, the whole-space
inequality \eqref{eq1.2}
implies that
$$
\| u \|_{L^r(\Omega)} \leq \| Eu \|_{L^r(\R^n)} \leq C \|\nabla Eu\|_{L^p
(\R^n)}^\theta \| Eu\|_{L^s(\R^n)}^{1-\theta} \leq C \| u \|_{W^{1,p}
(\Omega)}^\theta \| u \|_{L^s(\Omega)}^{1-\theta} \,.
$$
This is the claimed inequality for $s>1$.
The extension to $0<s\leq 1$ follows from the result of \cite{BCLS} and,
more precisely, from Theorems 3.1, 3.2, 3.3 and 3.4 in the cases $p<n$,
$p>n$, $p=n>1$ and $p=n=1$, respectively.
\end{proof}
\section{Proof of the generalized Keller--Lieb--Thirring inequality}\label
{sec:kellerliebthirring}
In this section we show, in the setting of Sobolev extension domains, that
the generalized Keller--Lieb--Thirring inequality is equivalent to a
Sobolev interpolation inequality.
\begin{proof}[Proof of Theorem \ref{thm1.5}]
Interchanging the infima over $u$ and $W$, we find
\begin{align*}
 & \inf_{0\neq W \in L^{\gamma+n/p}(\Omega)} \frac{E(W) + 1}{(\int_
\Omega W^
 {\gamma+n/p}\,dx)^{1/\gamma}} \\
 &= \inf_{0\not\equiv u\in W^{1,p}(\Omega)}
\inf_{w>0} w^{-\frac{\gamma+n/p}{\gamma}} \left( \frac{\|u\|_{W^{1,p}
 (\Omega)}^p}{\| u\|_{L^p(\Omega)}^p} - \sup_{\|W\|_{L^
{\gamma+n/p} (\Omega)}=w} \frac{\int_\Omega W |u|^p\,dx}{\| u
\|_{L^p(\Omega)}^p}
 \right).
\end{align*}
By H\"older's equality and its sharpness, we have
$$
\sup_{\|W\|_{L^{\gamma+n/p}(\Omega)}=w} \int_\Omega W |u|^p\,dx = w \| u
\|_{L^r(\Omega)}^p \,.
$$
Thus, we have shown that
$$
\inf_{0\neq W \in L^{\gamma+n/p}(\Omega)} \frac{E(W) + 1}{(\int_\Omega W^
{\gamma+n/p}\,dx)^{1/\gamma}} = \inf_{0\not\equiv u\in W^{1,p}(\Omega)}
\inf_{w>0} w^{-\frac{\gamma+n/p}{\gamma}} \left( \frac{\|u\|_{W^{1,p}
(\Omega)}^p}{\| u\|_{L^p(\Omega)}^p} - w \frac{\| u \|_{L^r(\Omega)}^p}{\|
u \|_{L^p(\Omega)}^p} \right).
$$
An elementary computation shows that, for any $A,B>0$,
$$
\inf_{w>0} w^{-\frac{\gamma+n/p}{\gamma}} \left( A - w B \right) = - \frac
{(n/(p\gamma))^{n/(p\gamma)}}{(1+n/(p\gamma))^{1+n/(p\gamma)}} A^{-\frac
{n/p}{\gamma}} B^{\frac{\gamma+n/p}{\gamma}}
$$
Thus, we obtain
\begin{align*}
& \inf_{0\neq W \in L^{\gamma+n/p}(\Omega)} \frac{E(W) + 1}{(\int_\Omega
W^
{\gamma+n/p}\,dx)^{1/\gamma}} \\ & = - \frac{(n/(p\gamma))^{n/(p\gamma)}}
{(1+n/(p\gamma))^{1+n/(p\gamma)}}
\sup_{0\not\equiv u\in W^{1,p}(\Omega)} \left( \frac{\|u\|_{W^{1,p}
(\Omega)}^p}{\| u\|_{L^p(\Omega)}^p} \right)^{-\frac{n/p}\gamma} \left(
\frac{\| u \|_{L^r(\Omega)}^p}{\| u \|_{L^p(\Omega)}^p} \right)^{\frac
{\gamma+n/p}{\gamma}} \\
& = - \frac{(n/(p\gamma))^{n/(p\gamma)}}{(1+n/(p\gamma))^{1+n/(p\gamma)}}
\left( S_r^\Omega \right)^{-\frac{p(\gamma+n/p)}{\gamma}}.
\end{align*}
This proves the claimed bound.
\end{proof}
\bibliographystyle{alpha}
\bibliography{bibliography}

\end{document}